\documentclass{amsart}
\usepackage[dvips]{graphicx}
\usepackage{amscd}
\usepackage{pstricks}
\theoremstyle{plain}
\newtheorem{theorem}{Theorem}

\newtheorem{definition}{Definition}
\newtheorem{corollary}{Corollary}

\theoremstyle{remark}
\newtheorem{remark}{Remark}

\numberwithin{equation}{section}

\numberwithin{definition}{section}

\begin{document}

\title[Topological Equivalence Nonautonomous difference equations]{Topological Equivalence of nonautonomous difference equations with a family of dichotomies on the half line}
\author[Casta\~neda]{\'Alvaro Casta\~neda}
\author[Gonz\'alez]{Pablo Gonz\'alez}
\author[Robledo]{Gonzalo Robledo}
\address{Departamento de Matem\'aticas, Facultad de Ciencias, Universidad de
  Chile, Casilla 653, Santiago, Chile}
\address{Facultad de Ingenier\'ia y Ciencias, Universidad Adolfo Iba\~nez, Pe\~nalol\'en, Santiago--Chile}  
\email{castaneda@uchile.cl}
\email{pablo.gonzalez.l@uai.cl}
\email{grobledo@uchile.cl}
\thanks{This work was supported by FONDECYT Regular 1170968}
\date{January 2020}

\begin{abstract}

A linear system of difference equations and a nonlinear perturbation are considered,
we obtain sufficient conditions to ensure the topological equivalence between them, namely,
the linear part satisfies a property of dichotomy on the positive half--line while the nonlinearity has some boundedness and Lipschitzness conditions. As a consequence, we study the asymptotical stability and its preservation by topological equivalence
\end{abstract}

\maketitle
\section{Introduction}
\subsection{Preliminaries} The linearization of flows arising from autonomous ordinary differential equations and  autonomous difference equations has a long history starting with the 
classical Hartman--Grobman Theorem \cite{Hartman60,Grobman}, which ensures the existence of a local homeomorfism 
between a nonlinear flow and its linearization around a fixed point, provided
that a hyperbolicity condition on the corresponding linearized flow is verified. The reader is refered to \cite{Palis,Pugh,Reinfelds1972} for an in depth look to the global case or an abstract setting.

The extension of the above results to the nonautonomous framework have dealed with 
the property of dichotomy \cite{Coffman,Coppel} which mimics some qualitative properties
of the hyperbolicity condition, namely, the existence of stable and unstable directions of a linear system; this fact has been useful to develop some local \cite{Dragicevic,Lin} and global linearization results. 

To the best of our knowledge, the global and nonautonomous linearization results started with 
the work of K.J. Palmer in \cite{Palmer}, which considered two systems of ordinary differential equations: a linear one and a nonlinear perturbation. Under the assumption that the linear system satisfies a uniform exponential dichotomy property \cite{Coppel} and meanwhile the nonlinear perturbation verifies some Lipschitzness and boundedness assumptions, it is proven that both systems are topologically equivalent; property that will be explained in full later on in this paper.

In order to obtain a discrete version of the Palmer's result, let us consider the non\-autonomous systems of difference equations
\begin{eqnarray}
\label{lineal}
\hspace{-1.6cm}x_{k+1}&=&A(k)x_{k},\quad k\in \mathbb{Z}^{+},
\end{eqnarray}
\begin{eqnarray}
\label{linealsystem}
y_{k+1}&=&A(k)y_{k}+f(k,y_{k}),\quad k\in \mathbb{Z}^{+}, \label{nolinealsystem}
\end{eqnarray}
where $x_{k}$ and $y_{k}$ are column vectors of $\mathbb{R}^{d}$ for any $k\in \mathbb{Z}^{+}:=\{0,1,2,\ldots\}$, the matrix function $k\mapsto A(k)\in M_{d}(\mathbb{R})$ is non singular and $f:\mathbb{Z}^{+}\times \mathbb{R}^{d}\to \mathbb{R}^{d}$ is continuous in $\mathbb{R}^{d}$.  

The purpose of this article is to obtain a set of conditions ensuring that the above systems are
topologically equivalent, this property was introduced in the continuous framework by K.J. Palmer in \cite{Palmer} and extended to the discrete case by several authors such as G. Papaschinopoulos and J. Schinas in \cite{Papas87,Schinas} who stated as follows: 
\begin{definition}
\label{TopEq}
Let $J\subseteq \mathbb{Z}$. The systems \textnormal{(\ref{lineal})} and \textnormal{(\ref{nolinealsystem})} are $J$--topologically e\-qui\-valent if there exists a function such as $H\colon J \times \mathbb{R}^{d}\to \mathbb{R}^{d}$ with the properties
\begin{itemize}
\item[(i)] If $x(k)$ is a solution of \textnormal{(\ref{lineal})}, then $H[k,x(k)]$ is a solution 
of \textnormal{(\ref{nolinealsystem})},
\item[(ii)] $H(k,u)-u$ is bounded in $J \times \mathbb{R}^{d}$,
\item[(iii)]  For each fixed $k\in J$, $u\mapsto H(k,u)$ is an homeomorphism of $\mathbb{R}^{d}$.
\end{itemize}
In addition, the function $u\mapsto G(k,u)=H^{-1}(k,u)$ has properties \textnormal{(ii)--(iii)} and maps solutions of \textnormal{(\ref{nolinealsystem})} into solutions of \textnormal{(\ref{lineal})}.
\end{definition}

The property of topological equivalence has several differences with the linearization arising from the classical Hartman--Grobman's theorem: i) it is inserted in a nonautonomous framework and there is not an univocal equivalent to the hyperbolicity condition, ii) it deals with a global linearization instead of a local one, iii) an explicit construction of the homeomorphisms is possible in some cases, iv) the smoothness properties are considerably less studied, v) a corresponding version of the resonance's condition is far from being completed.

The $\mathbb{Z}$--topological equivalence between (\ref{lineal}) and (\ref{nolinealsystem}) has been studied in several works inspired in the Palmer's approach. First of all, A. Reinfelds in \cite{Reinfelds1997,Reinfelds} obtained a topological equivalence result by assuming that (\ref{lineal}) has a dichotomy and by constructing the homeomorphisms $H$ and $G$ based in the Green's function associated to the dichotomy combined with technical conditions on the nonlinear part. Secondly, we make the point of mentioning the work of G. Papaschinopoulos \cite{Papaschinopoulos} which studied the topological equivalence in a continuous/discrete framework and the discrete case is studied as a technical step. We also mention the work \cite{Castaneda2016}, where the authors obtained a $\mathbb{Z}$--topological equivalence result by considering a generalized exponential dichotomy in the linear part combined with the Reinfelds's assumptions on the nonlinearities and the continuity of $G$ and $H$ is addressed in detail. We also highlight a related result from L. Barreira and C. Valls \cite{Barreira2006} which is not exactly a topological equivalence but considers
a linear part with a nonuniform exponential dichotomy on $\mathbb{Z}$ and obtained properties of H\"older regularity on the corresponding homeomorphisms.

We point out that there exist other linearization results which follow ideas and methods different to 
the Palmer's construction. In particular, we highligth the approach based in the crossing times with the unit ball which has been employed with several variations in \cite{Barreira2016,Castaneda2018,Kurzweil}.

\subsection{Notation} Throughout this paper, the symbols $|\cdot|$ and $||\cdot||$ will denote respectively a vector norm and its induced matrix norm. The Banach space of bounded sequences from $\mathbb{Z}^{+}$ to
$\mathbb{R}^{d}$ will be denoted by $\ell^{\infty}(\mathbb{Z}^{+},\mathbb{R}^{d})$ with supremum norm 
$|\cdot|_{\infty}$.

\begin{definition}
A fundamental matrix of the system $(\ref{lineal})$ is a matrix function $\Phi\colon \mathbb{Z}^{+}\to M_{d}(\mathbb{R})$ such that its columns are a basis of solutions of $(\ref{lineal})$ and satisfies the matrix difference equation
\begin{displaymath}
\Phi(n+1)=A(n)\Phi(n).
\end{displaymath}
\end{definition}
\begin{definition}
The transition matrix of $(\ref{lineal})$ is defined by:
\begin{eqnarray}
\label{transitionmatrix1}
\Phi(k,n)=\left\{
\begin{array}{ccl}
 A(k-1)A(k-2)\cdots A(n), &\mbox{if}& k>n, \\\\
 I, &\mbox{if}& k=n, \\\\ 
 A^{-1}(k)A^{-1}(k+1)\cdots A^{-1}(n-1) & \mbox{if}& \: k<n. \end{array}
\right.
\end{eqnarray}
\end{definition}

\subsection{Novelty of this work} Our work is inscribed in the context of Palmer's approach 
considered previously in \cite{Castaneda2016,Papaschinopoulos,Reinfelds1997,Reinfelds} but has some differences that will be explained below.

First of all and contrarily to the previous references, we obtain a result of topological equivalence with $J=\mathbb{Z}^{+}$ instead of $J=\mathbb{Z}$. This fact induced technical differences and additional difficulties when constructing the maps $G$ and $H$ mainly due to the lack of admissibility results.

Secondly, we have obtained more detailed characterizations of the maps $G$ and $H$. In particular, we emphasize the remarkable simplicity for the map $u\mapsto G(k,u)$, this fact allow us to deduce some nice and new identities for $G$.

Finally, when we restrict our attention to the case when the linear system is (\ref{lineal}) is asymptotically stable, we obtain sharper results of $\mathbb{Z}^{+}$--topological equivalence which allows a simpler study about the smoothness properties of $G$ and $H$ 
and to prove that the asymptotic stability is preserved by the equivalence when the nonlinear system has an equilibrium.

\section{Main Result}

\subsection{Statement} In order to state the main result, we will assume that the linear system $(\ref{lineal})$ satisfies the following pro\-per\-ties:
\begin{itemize}
    \item[{\bf (P1)}] The matrix function $k\mapsto A(k)$ is invertible and  uniformly bounded, that is, there exists $M \geq1$ such that
    $$
    \displaystyle \max\left\{\sup_{k\in \mathbb{Z}^{+}}||A(k)||,\sup_{k\in \mathbb{Z}^{+}}||A^{-1}(k)||\right\}=M.
    $$
    \item[{\bf (P2)}] The linear system
    (\ref{lineal}) has a nonuniform dichotomy. That is, there exists two invariant projectors $P(\cdot)$ and $Q(\cdot)$ such that $P(n)+Q(n)=I(n)$ for any $n\in \mathbb{Z}^{+}$,
    a bounded sequence $\rho$ 
 and a decreasing sequence $h$ convergent to zero with $h(0)=1$ such that:
    \begin{displaymath}
    \left\{\begin{array}{rl}
    ||\Phi(k,n)P(n)||\leq \displaystyle \rho(n)\Big(\frac{h(k)}{h(n)}\Big), & \forall k\geq n \geq 0\\\\
    ||\Phi(k,n)Q(n)||\leq \displaystyle \rho(n)\Big(\frac{h(n)}{h(k)}\Big), & \forall 0\leq k \leq n.
    \end{array}\right.
    \end{displaymath}
\end{itemize}    
 
 \begin{remark}
The assumption \textbf{(P1)} is only techical and implies that the linear system (\ref{lineal}) has the property of bounded growth on $\mathbb{Z}^{+}$, that is, 
$$
||\Phi(k,\ell)||\leq M^{|k-\ell|} \quad \textnormal{for any $k,\ell\in \mathbb{Z}^{+}$}
$$
and we refer the reader to \cite{Aulbach} and \cite{Coppel} for details.
\end{remark}

\begin{remark}
The assumption \textbf{(P2)} can be seen as a nonautonomous version of the hyperbolicity property of the autonomous case. 

\noindent i) When $\rho(n)=K>0$ and $h(n)=\theta^{n}$ with $\theta\in (0,1)$ for any $n\in \mathbb{Z}^{+}$, \textbf{(P2)} means that the system (\ref{lineal}) has the property of uniform exponential dichotomy on $J=\mathbb{Z}^{+}$. This property and some of its consequences has been extensively studied in \cite{Kurzweil,Papas85,Papas87,Papas88}.

\noindent ii) When $\rho(n)=K>0$ and 
$h(n)=\exp(-\sum_{j=0}^{n}u_{j})$, where the sequence $u_{j}$ is positive and non 
 summable, \textbf{(P2)} means that the system (\ref{lineal}) has the property of generalized exponential dichotomy on $J=\mathbb{Z}^{+}$. This property and its applications in topological equivalence has been studied in \cite{Castaneda2016} for the case $J=\mathbb{Z}$. 

\noindent iii) When considering assumptions different and/or more general than those stated by \textbf{(P2)}
we can obtain other dichotomies and we mention some as the $(h,k)$--dichotomies \cite{Crai}, the  nonuniform exponential dichotomy \cite{Megan2016,Barreira2016}, the $(\mu,\nu)$-- nonuniform dichotomies \cite{Bento} and and the polynomial dichotomies \cite{BarreiraTMNA}.

\end{remark}

\begin{remark}
\label{stabilite}
The property \textbf{(P2)} restricted to the particular case $P(n)=I$ and $Q(n)=0$
implies that the origin is an asymptotically stable equilibrium of (\ref{lineal}). In addition, \textbf{(P2)} allows us to characterize several types of asymptotic stabilities. It is well known that the
uniform asymptotical stability is verified if and only if (\ref{lineal}) admits a uniform exponential dichotomy on $\mathbb{Z}^{+}$ with $P(n)=I$. The dichotomy property allows us to describe another type of asymptotic stabilities more general than the uniform one which are described in terms of $(\rho, h)-$contractions (see \cite{Chu} for details).
\end{remark}

\begin{remark}
The projectors $P$ and $Q$ are named invariant since they satisfy
$$
P(k)\Phi(k,n)=\Phi(k,n)P(n) \quad \textnormal{and} \quad Q(k)\Phi(k,n)=\Phi(k,n)Q(n) \quad \textnormal{for any $k,n\in \mathbb{Z}^{+}$}
$$
and finally, \textbf{(P2)} allows to define the Green's function $\mathcal{G}\colon \mathbb{Z}^{+}\times \mathbb{R}^{d}\to \mathbb{R}^{d}$ associated to the linear system (\ref{lineal}) as follows
\begin{displaymath}
\mathcal{G}(k,n)=\left\{\begin{array}{rl}
    \Phi(k,n)P(n) & \forall \, k\geq n \geq 0,\\\\
    -\Phi(k,n)Q(n) & \forall \,0\leq k< n.
    \end{array}\right.
\end{displaymath}

\end{remark}    

Moreover, we will assume that the nonlinear system $(\ref{nolinealsystem})$ has a perturbation $f$ 
that satisfies the following properties 

    \begin{itemize}    
    \item[{\bf (P3)}] For any $k\in \mathbb{Z}^{+}$ and any couple $(y,\tilde{y})\in \mathbb{R}^{d}\times \mathbb{R}^{d}$ it follows that:
    $$
    |f(k,y)-f(k,\tilde{y})|\leq \gamma(k) |y-\tilde{y}|\quad \textnormal{and} \quad |f(k,y)|\leq \mu(k).
    $$
    \item[{\bf{(P4)}}] The sequence $\mu$ defined above verifies
\begin{displaymath}
N(\ell,\mu)=\sum\limits_{j=0}^{\infty}|\mathcal{G}(\ell,j+1)\mu(j)|=p<+\infty \quad \textnormal{for any $\ell\in \mathbb{Z}^{+}$}.
\end{displaymath}

\item[{\bf{(P5)}}] The sequence $\gamma$ defined above verifies
\begin{displaymath}
N(\ell,\gamma)=\sum\limits_{j=0}^{\infty}|\mathcal{G}(\ell,j+1)\gamma(j)|=q<1 \quad \textnormal{for any $\ell\in \mathbb{Z}^{+}$}.
\end{displaymath}
\item[{\bf{(P6)}}] The sequence $\gamma(\cdot)$ and $A(\cdot)$ are such that
\begin{displaymath}
|A^{-1}(\ell)\gamma(\ell)|<1 \quad \textnormal{for any $\ell\in \mathbb{Z}^{+}$}.
\end{displaymath}
\end{itemize}

\begin{remark}
The properties \textbf{(P3)}--\textbf{(P5)} have been used previously in the study of the topological equivalence problem in \cite{Reinfelds} and later on \cite{Castaneda2016}. These properties allowed to generalize the construction of the homeomorphisms when the linear system (\ref{lineal}) has dichotomies more general than the exponential one. Nevertheless, there exists a trade off between the assumptions on the linear part and the nonlinear perturbation. More general dichotomies on the linear part induced more restrictive assumptions on the perturbations. 
\end{remark}

\begin{remark}
\label{R6}
The property \textbf{(P6)} is a technical assumption and ensures that 
any solution $n\mapsto y(n,k,\eta)$ of the nonlinear system (\ref{nolinealsystem}) 
passing through $\eta$ at $n=k$ can be backward continued for any $n\in \{0,\ldots,k-1\}$. If fact, notice that $y(k-1,k,\eta)$ can be seen as the unique fixed point of the map
$\Theta_{k-1}\colon \mathbb{R}^{d}\to \mathbb{R}^{d}$ defined by $\Theta_{k}(u)=A^{-1}(k-1)\eta-A^{-1}(k-1)f(k-1,u)$ and the terms $y(n,k,\eta)$ with $n\in \{0,\ldots,k-2\}$ can be obtained in a similar way.

We must point out that \textbf{(P5)} implies that
\begin{displaymath}
|A^{-1}(\ell)Q(\ell)\gamma(\ell)|<q-\sum\limits_{\substack{j=0\\j\neq \ell}}^{\infty}|\mathcal{G}(\ell,j+1)\gamma(j)|<q<1.
\end{displaymath}
\end{remark}

As we have set forth the premises now we are able to state our main result 
\begin{theorem}
\label{main}
If the assumptions {\bf (P1)}-{\bf (P6)} are satisfied 
then the systems $(\ref{lineal})$ and $(\ref{nolinealsystem})$ are $\mathbb{Z}^{+}-$ topologically equivalent.
\end{theorem}

\begin{proof}
The proof of this result will be made in several steps.

\noindent\textit{Step 1: Preliminaries.} Let $k\mapsto x(k,m,\xi)$ and $k\mapsto y(k,m,\eta)$ be the respective solutions of the systems $(\ref{lineal})$ and $(\ref{nolinealsystem})$ 
with initial conditions $\xi$ and $\eta$ at $k=m$. 

Now, let us introduce the map:
\begin{equation}
\label{map1}
\begin{array}{rcl}
w^{*}(k;(m,\eta))&=&\displaystyle -\sum\limits_{j=0}^{\infty}\mathcal{G}(k,j+1)f(j,y(j,m,\eta)),\\\\
&=&
\displaystyle -\sum\limits_{j=0}^{k-1}\Phi(k,j+1)P(j+1)f(j,y(j,m,\eta))\\\\
&&\displaystyle  +\sum\limits_{j=k}^{\infty}\Phi(k,j+1)Q(j+1)f(j,y(j,m,\eta)),
\end{array}
\end{equation}
and the map $\Gamma_{(m,\xi)}\colon \ell^{\infty}(\mathbb{Z}^{+},\mathbb{R}^{d}) \to  \ell^{\infty}(\mathbb{Z}^{+},\mathbb{R}^{d})$ defined by
\begin{displaymath}
\begin{array}{rcl}
\Gamma_{(m,\xi)}\phi(k)&=& \displaystyle
\sum\limits_{j=0}^{+\infty}\mathcal{G}(k,j+1)f(j,x(j,m,\xi)+\phi(j)),
\\\\
&=&\displaystyle \sum\limits_{j=0}^{k-1}\Phi(k,j+1)P(j+1)f(j,x(j,m,\xi)+\phi(j))\\\\
&& \displaystyle
-\sum\limits_{j=k}^{\infty}\Phi(k,j+1)Q(j+1)f(j,x(j,m,\xi)+\phi(j)).
\end{array}
\end{displaymath}

Take notice that the backward continuation of the solutions of the nonlinear system (\ref{nolinealsystem}) is necessary to ensure
that the map $k\mapsto w^{*}(k;(m,\eta))$ is well defined; this is provided by \textbf{(P6)} and Remark \ref{R6}.

Let $\phi,\psi\in \ell^{\infty}(\mathbb{Z}^{+},\mathbb{R}^{d})$. By using \textbf{(P2)},\textbf{(P3)} and \textbf{(P4)} we can note that
\begin{displaymath}
\begin{array}{rcl}
\left|\Gamma_{(m,\xi)}\phi(k)-\Gamma_{(m,\xi)}\psi(k)\right|&\leq & \displaystyle \sum\limits_{j=0}^{+\infty}\left|\mathcal{G}(k,j+1)\{f(j,x(j,m,\xi)+\phi(j))-f(j,x(j,m,\xi)+\psi(j))\}\right|\\\\
&\leq&  \displaystyle \sum\limits_{j=0}^{+\infty}\gamma(j)||\mathcal{G}(k,j+1)||\,|\phi(j))-\psi(j))|\\\\
&\leq& q\,|\phi-\psi|_{\infty}
\end{array}
\end{displaymath}
and by using the Banach contraction principle we have the existence of a unique fixed point
\begin{equation}
\label{FP}
z^{*}(k;(m,\xi))=\sum\limits_{j=0}^{+\infty}\mathcal{G}(k,j+1)f(j,x(j,m,\xi)+z^{*}(j;(m,\xi))).
\end{equation}

It is easy to verify that the maps $k\mapsto w^{*}(k;(m,\eta))$ and $k\mapsto z^{*}(k;(m,\xi))$ are solutions of the
initial value problems:
\begin{displaymath}
\left\{\begin{array}{cll}
     w_{k+1}&=&A(k)w_{k}-f(k,y(k,m,\eta))  \\\\
     w_{0}&=& \displaystyle -\sum\limits_{j=0}^{\infty}\Phi(0,j+1)Q(j+1)f(j,y(j,m,\eta)).
\end{array}\right.
\end{displaymath}
and 
\begin{displaymath}
\left\{\begin{array}{cll}
     z_{k+1}&=&A(k)z_{k}+f(k,x(k,m,\xi)+z_{k})  \\
     z_{0}&=& \displaystyle  \sum\limits_{j=0}^{+\infty}\Phi(0,j+1)Q(j+1)f(j,x(j,m,\xi)+z^{*}(j;(m,\xi)))
\end{array}\right.
\end{displaymath}

\noindent\textit{Step 2: Constructing $H$ and $G$.} By the uniqueness of solutions we have that
\begin{equation}
\label{semiflujo-lineal}
x(k,m,\xi)=x(k,p,x(p,m,\xi)) \quad \textnormal{for any $k,p,m\in \mathbb{Z}^{+}$},
\end{equation}
and the reader can verify that
\begin{equation}
\label{identidad-dif}
z^{*}(k;(m,\xi))=z^{*}(k;(p,x(p,m,\xi))) \quad \textnormal{for any $k,p,m\in \mathbb{Z}^{+}$}.
\end{equation}

For any fixed $k\in \mathbb{Z}^{+}$, let us construct the maps $H(k,\cdot)\colon \mathbb{R}^{d}\to \mathbb{R}^{d}$
and  $G(k,\cdot)\colon \mathbb{R}^{d}\to \mathbb{R}^{d}$ as follows:

\begin{equation}
\label{Homeo-H}
\left\{
\begin{array}{rcl}
H(k,\xi) & = & \xi +\sum\limits_{j=0}^{+\infty}\mathcal{G}(k,j+1)f(j,x(j,k,\xi)+z^{*}(j;(k,\xi))) \\\\
         & = & \xi+ z^{*}(k;(k,\xi)),
\end{array}\right.
\end{equation}
and
\begin{equation}
\label{Homeo-G}
\left\{
\begin{array}{rcl}
G(k,\eta) & = & \eta-\sum\limits_{j=0}^{+\infty}\mathcal{G}(k,j+1)f(j,y(j,k,\eta))\\\\
          & = & \eta + w^{*}(k;(k,\eta)).
\end{array}\right.
\end{equation}

As $k\mapsto z^{*}(k;(k,\xi))$ and $k\mapsto w^{*}(k;(k,\eta))$ are uniformly bounded sequences,
it follows that both $H$ and $G$ satisfy the statement (ii) from the Definition \ref{TopEq}. Now, in order 
to study some additional properties of $G$, let us consider the initial value problem:
\begin{eqnarray}
\label{PVInolinealsyestem}
\left\{\begin{array}{ccl}
    y_{n+1} & = &A(n)y_{n}+f(n,y_{n})  \\
    y_{k}&= & \eta.
\end{array}\right.
\end{eqnarray}
If $n<k$ we have that:
$$
y(n,k,\eta)=\Phi(n,k)\eta-\sum_{j=n}^{k-1}\Phi(n,j+1)f(j,y(j,k,\eta)),
$$
which is equivalent to:
\begin{displaymath}
\begin{array}{rcl}
\Phi(k,n)y(n,k,\eta)&=&\displaystyle  \eta-\sum_{j=n}^{k-1}\Phi(k,j+1)f(j,y(j,k,\eta)),\\\\
&=&\displaystyle  \eta-\sum_{j=n}^{k-1}\Phi(k,j+1)\{P(j+1)+Q(j+1)\}f(j,y(j,k,\eta)),\\\\
&=&\displaystyle
\eta-\sum_{j=n}^{k-1}\Phi(k,j+1)P(j+1)f(j,y(j,k,\eta))-\sum_{j=n}^{k-1}\Phi(k,j+1)Q(j+1)f(j,y(j,k,\eta)).
\end{array}
\end{displaymath}

In particular, if $n=0$ we have that
\begin{displaymath}
\begin{array}{rcl}
\Phi(k,0)y(0,k,\eta)&=&\displaystyle
\eta-\sum_{j=0}^{k-1}\Phi(k,j+1)P(j+1)f(j,y(j,k,\eta))-\sum_{j=0}^{k-1}\Phi(k,j+1)Q(j+1)f(j,y(j,k,\eta)),\\\\
&=&\displaystyle
\eta-\sum_{j=0}^{k-1}\Phi(k,j+1)P(j+1)f(j,y(j,k,\eta))+\sum_{j=k}^{\infty}\Phi(k,j+1)Q(j+1)f(j,y(j,k,\eta))\\\\
& & \displaystyle
-\sum_{j=0}^{\infty}\Phi(k,j+1)Q(j+1)f(j,y(j,k,\eta)),\\\\
&=& \displaystyle \eta- \sum\limits_{j=0}^{\infty}\mathcal{G}(k,j+1)f(j,y(j,k,\eta))-\sum_{j=0}^{\infty}\Phi(k,j+1)Q(j+1)f(j,y(j,k,\eta)),\\\\
&=&\displaystyle G(k,\eta)-\Phi(k,0)\sum_{j=0}^{\infty}\Phi(0,j+1)Q(j+1)f(j,y(j,k,\eta)).
\end{array}
\end{displaymath}

Now, by using the definition of the map $n\mapsto w^{*}(0;(k,\eta))$ we can deduce that
\begin{equation}
\label{homeo-G1}
\begin{array}{rcl}
G(k,\eta)&=&\displaystyle \Phi(k,0)\{y(0,k,\eta)+\sum_{j=0}^{\infty}\Phi(0,j+1)Q(j+1)f(j,y(j,k,\beta))\}\\\\
&=&\displaystyle \Phi(k,0)\{y(0,k,\eta)+w^{*}(0;(k,\eta))\}
\end{array}
\end{equation}

\noindent\textit{Step 3: $H$ maps solutions of (\ref{lineal}) into solutions of (\ref{nolinealsystem})
and $G$ maps solutions of (\ref{nolinealsystem}) into solutions of (\ref{lineal}).}  By using (\ref{semiflujo-lineal}), (\ref{identidad-dif}) and (\ref{Homeo-H}) we can see that
\begin{displaymath}
\begin{array}{rcl}
H[k,x(k,m,\xi)] & = & x(k,m,\xi) +\sum\limits_{j=0}^{\infty}\mathcal{G}(k,j+1)f(j,x(j,m,\xi)+z^{*}(j;(m,\xi)))) \\\\
 & = & x(k,m,\xi) +z^{*}(k;(m,\xi)),
\end{array}
\end{displaymath}
which has the alternative description of
\begin{displaymath}
H[k,x(k,m,\xi)]  =  x(k,m,\xi) +\sum\limits_{j=0}^{\infty}\mathcal{G}(k,j+1)f(j,H[j,x(j,m,\xi)]).
\end{displaymath}

The above identities combined with (\ref{lineal}), (\ref{FP}) and 
$\mathcal{G}(k+1,j+1)=A(k)\mathcal{G}(k,j)$ allows us to prove that
\begin{displaymath}
\begin{array}{rcl}
H[k+1,x(k+1,m,\xi)] & = & x(k+1,m,\xi)+z^{*}(k+1;(m,\xi)) \\\\
                    & = & A(k)\left\{x(k,m,\xi)+z^{*}(k;(m,\xi)))\right\}+f(k,x(k,m,\xi)+z^{*}(k;(m,\xi)))\\\\
                    & = & A(k)H[k,x(k,m,\xi)]+f(k,H[k,x(k,m,\xi)]),
\end{array}
\end{displaymath}
at this point we conclude that $k\mapsto H[k,x(k,m,\xi)]$ is solution of $(\ref{nolinealsystem})$ passing through $H(m,\xi)$ at $k=m$. In addition, 
as consequence of uniqueness of solution we obtain that
\begin{displaymath}
H[k,x(k,m,\xi)]=y(k,m,H(m,\xi)).
\end{displaymath}

We can summarize several characterizations of $H[k,x(k,m,\eta)]$:
\begin{equation}
\label{identity-HY}
H[k,x(k,m,\xi)]=\left\{
\begin{array}{l}
x(k,m,\xi) +z^{*}(k;(m,\xi))\\\\
 x(k,m,\xi) +\sum\limits_{j=0}^{\infty}\mathcal{G}(k,j+1)f(j,H[j,x(j,m,\xi)])\\\\
 y(k,m,H(m,\xi))
 \end{array}\right.
\end{equation}

Similarly, the uniqueness of solutions implies the identity
\begin{equation}
\label{semiflujo-nolin}   
y(k,m,\eta)=y(k,p,y(p,m,\eta)) \quad \textnormal{ for any $k,p,m\in \mathbb{Z}^{+}$},
\end{equation}
which allows us to deduce that
\begin{equation}
\label{identidad-dif2}
w^{*}(k;(m,\eta))=w^{*}(k;(p,y(p,m,\eta))) \quad \textnormal{for any $k,p,m\in \mathbb{Z}^{+}$}.
\end{equation}

Thus from the previous expression it follows that
\begin{eqnarray*}
G[k,y(k,m,\eta)]&=&y(k,m,\eta)-\sum_{j=0}^{\infty}\mathcal{G}(k,j+1)f(j,y(j,m,\eta)),\\
                &=&y(k,m,\eta)+w^{*}(k;(m,\eta)),
\end{eqnarray*}
now let us note that
\begin{eqnarray*}
G[k+1,y(k+1,m,\eta)]&=&y(k+1,m,\eta)+w^{*}(k+1;(m,\eta)),\\
                    &=&A(k)\{y(k,m,\eta)+w^{*}(k;(m,\eta))\}+f(k,y(k,m,\eta))-f(k,y(k,m,\eta)),\\
                    &=&A(k)G[k,y(k,m,\eta)]
\end{eqnarray*}
then $k\mapsto G[k,y(k,m,\eta)]$ is solution of $(\ref{lineal})$ passing through $G(m,\eta)$ at $k=m.$ 

\bigskip 
In addition, since $k\mapsto G[k,y(k,m,\eta)]$ is solution of $(\ref{lineal})$ passing through $G(m,\eta)$ at $k=m$, then we have that
\begin{displaymath}
G[k,y(k,m,\eta)]=x(k,m,G(m,\eta))=\Phi(k,m)G(m,\eta),
\end{displaymath}
which also has an alternative formulation by using (\ref{homeo-G1}) and (\ref{identidad-dif2}):
$$
G[k,y(k,m,\eta)]=\Phi(k,0)\{y(0,m,\eta)+w^{*}(0;(m,\eta))\}.
$$
We are also available to summarize several characterizations of $G[k,y(k,m,\eta)]$:
\begin{equation}
\label{identity-GX}
G[k,y(k,m,\eta)]=\left\{\begin{array}{l}
y(k,m,\eta)+w^{*}(k;(m,\eta))\\\\
x(k,m,G(m,\eta))=\Phi(k,m)G(m,\eta)\\\\
\Phi(k,0)\{y(0,m,\eta)+w^{*}(0;(m,\eta))\}.
\end{array}\right.
\end{equation}


\noindent\textit{Step 4: $u\mapsto G(k,u)$ and $u\mapsto H(k,u)$ are bijective for any fixed $k\in \mathbb{Z}^{+}$.} 

\bigskip

By using the description of $H[k,x(k,m,\xi)]$ combined with identities (\ref{semiflujo-nolin}) and (\ref{identity-HY}) we can deduce that
\begin{displaymath}
\begin{array}{rcl}
G[k,H[k,x(k,m,\xi)]] & = & H[k,x(k,m,\xi)] \\\\
&    &  \displaystyle
-\sum\limits_{j=0}^{\infty}\mathcal{G}(k,j+1)f(j,y(j,k,H[k,x(k,m,\xi)\, ]\, )\, )\\\\
& = & \displaystyle x(k,m,\xi)+\sum\limits_{j=0}^{\infty}\mathcal{G}(k,j+1)f(j,H[j,x(j,m,\xi)\, ]\, )\\\\
&   &\displaystyle -\sum\limits_{j=0}^{\infty}\mathcal{G}(k,j+1)f(j,y(j,k,H[k,x(k,m,\xi)\, ]\,)\,)\\\\
& = & x(k,m,\xi).
\end{array}
\end{displaymath}




In order to study $H[k,G[k,y(k,m,\eta)]]$, the first identity of (\ref{identity-HY}) allows us to verify that
\begin{displaymath}
\begin{array}{rcl}
 H[j,x(j,k,G[k,y(k,m,\eta)])] &=&  x(j,k,G[k,y(k,m,\eta)])+z^{*}(j;(k,G[k,y(k,m,\eta)]))\\\\ 
&=:& L[j,y(k,m,\eta)].
\end{array}
\end{displaymath}

On the other hand, by using (\ref{semiflujo-lineal}), (\ref{identidad-dif}), (\ref{identity-HY}) and (\ref{identity-GX}) it can be proved that
\begin{equation}
\label{IU}
\begin{array}{rcl}
H[j,x(j,k,G[k,y(k,m,\eta)])]
&=&
x(j,m,G(m,\eta))+z^{*}(j;(m,G(m,\eta))) \\\\
&=& H[j,x(j,m,G(m,\eta))]\\\\
&=& H[j,G[j,y(j,m,\eta)]].
\end{array}
\end{equation}

At this juncture, we have that:
\begin{eqnarray*}
H[k,G[k,y(k,m,\eta)]] &=&G[k,y(k,m,\eta)]+\sum_{j=0}^{\infty}\mathcal{G}(k,j+1)f(j,L[j,u(k,m,\eta)])\\\\
                    &=&y(k,m,\eta)-\sum_{j=0}^{\infty}\mathcal{G}(k,j+1)f(j,y(j,k,y(k,m,\eta))\\\\
                    &&+\sum_{j=0}^{\infty}\mathcal{G}(k,j+1)f(j,L[j,y(k,m,\eta)])\\\\
                    &=&y(k,m,\eta)-\sum_{j=0}^{\infty}\mathcal{G}(k,j+1)\{f(j,y(j,m,\eta))-f(j,L[j,y(k,m,\eta)])\}.
\end{eqnarray*}

Now, let us define
$$
w(k)=\left|H[k,G[k,y(k,m,\eta)]]-y(k,m,\eta)\right|.
$$

Then by using the above inequalities combined with (\ref{IU}) we have that
\begin{eqnarray*}
w(k)&\leq &\sum_{j=0}^{\infty}\left|\mathcal{G}(k,j+1)\right|\, |f(j,x(j,L[j,y(k,m,\eta)])-f(j,y(j,m,\eta))|,\\\\
    &\leq &\sum_{j=0}^{\infty}\gamma(j) \left|\mathcal{G}(k,j+1)\right|\, |H[j,G(j,y(j,m,\eta))]-y(j,m,\eta)|,\\\\
    &\leq & \sum_{j=0}^{\infty} \left|\mathcal{G}(k,j+1)\right|\,\gamma(j)w(j).
\end{eqnarray*}

By \textbf{(P4)} we know that $w\in \ell^{\infty}(\mathbb{Z}^{+},\mathbb{R}^{d})$ 
and by \textbf{(P5)} it follows that $|w|_{\infty}\leq q\,|w|_{\infty}$. Therefore if $w>0$ then
$1\leq q$, obtaining a contradiction. Hence $w(k)=0$ for any $k\in \mathbb{Z}^{+}$ and therefore we have
$$
H[k,G[k,y(k,m,\eta)]]=y(k,m,\eta),\quad \forall k\in \mathbb{Z}^{+}.
$$

In particular, if $k=m$ then
$$
H(m,G(m,\eta))=\eta,
$$
and hence we conclude that $u\mapsto H(k,u)$ is a bijection for any $k\in \mathbb{Z}^{+}$
and $u\mapsto G(k,u)$ is its inverse.

\bigskip

\noindent\textit{Step 5: $G$ is a continuous map.} First, if $k\in \mathbb{Z}^{+}$ is fixed, from the expression $(\ref{homeo-G1})$ we have
\begin{eqnarray*}
|G(k,\eta)-G(k,\tilde{\eta})|\leq ||\Phi(k,0)||\, \left\{|y(0,k,\eta)-y(0,k,\tilde{\eta})|+|w^{*}(0;(k,\eta))-w^{*}(0;(k;\tilde{\eta}))|\right\}.
\end{eqnarray*}

We will make separate estimations 
of the terms
\begin{equation}
\label{conP5fA}
|y(0,k,\eta)-y(0,k,\tilde{\eta})|
\quad \textnormal{and} \quad |w^{*}(0;(k,\eta))-w^{*}(0;(k;\tilde{\eta}))|.
\end{equation}

In order to study the first term of (\ref{conP5fA}), we can see that if $n<k$ then
\begin{equation}
\label{dif-1}
y(n,k,\eta)=\Phi(n,k)\eta-\sum_{j=n}^{k-1}\Phi(n,j+1)f(j,y(j,k,\eta)),
\end{equation}
thus (see Remark \ref{R6} for details) we have that
$$
y(k-1,k,\eta)=A^{-1}(k-1)\eta-A^{-1}(k-1)f(k-1,y(k-1,k,\eta)),
$$
then
\begin{eqnarray*}
|y(k-1,k,\eta)-y(k-1,k,\tilde{\eta})|&\leq & ||A^{-1}(k-1)||\, |\eta-\tilde{\eta}|\\\\
&&+||A^{-1}(k-1)\gamma(k-1)||\, |y(k-1,k,\eta)-y(k-1,k,\tilde{\eta})|,
\end{eqnarray*}
which implies by \textbf{(P6)} that
\begin{eqnarray}
\label{antigronwall}
&&|y(k-1,k,\eta)-y(k-1,k,\tilde{\eta})|\leq \frac{||A^{-1}(k-1)||}{\left(1-||A^{-1}(k-1)\gamma(k-1)||\right)}|\eta-\tilde{\eta}|.
\end{eqnarray}

Similarly, it follows that
$$
|y(k-2,k,\eta)-y(k-2,k,\tilde{\eta})|\leq \frac{||A^{-1}(k-2)||}{\left(1-||A^{-1}(k-2)\gamma(k-2)||\right)}|y(k-1,k,\eta)-y(k-1,k,\tilde{\eta})|,
$$
and from $(\ref{antigronwall})$ we have that
$$
|y(k-2,k,\eta)-y(k-2,k,\tilde{\eta})|\leq \prod_{p=k-2}^{k-1}\frac{||A^{-1}(p)||\, }{\left(1-||A^{-1}(p)\gamma(p)||\right)}|\eta-\tilde{\eta}|.
$$

Hence, inductively we can deduce that
$$
|y(k-j,k,\eta)-y(k-j,k,\tilde{\eta})|\leq \prod_{p=k-j}^{k-1}\frac{||A^{-1}(p)||\, }{\left(1-||A^{-1}(p)\gamma(p)||\right)}|\eta-\tilde{\eta}|.
$$

Now, if $n<k$ then there exists $j\in \mathbb{Z}^{+}$ such that $n+j=k.$ Thus
\begin{equation}
\label{conP5rev}
|y(n,k,\eta)-y(n,k,\tilde{\eta})|\leq \mathcal{C}_{k}(n)|\eta-\tilde{\eta}|,\quad \textnormal{for all $n<k$},
\end{equation}
where
$$
\mathcal{C}_{k}(n)=\prod_{j=n}^{k-1}\frac{||A^{-1}(j)||\, }{\left(1-||A^{-1}(j)\gamma(j)||\right)}.
$$

In particular
\begin{eqnarray}
\label{conP5revparticular}
|y(0,k,\eta)-y(0,k,\tilde{\eta})|\leq \mathcal{C}_{k}(0)|\eta-\tilde{\eta}|.
\end{eqnarray}

Now, by {\bf (P4)} we have that the second term of (\ref{conP5fA}) is estimated as follows:
\begin{eqnarray*}
|w^{*}(0;(k,\eta))-w^{*}(0;(k,\tilde{\eta}))|&\leq & \sum_{j=0}^{\infty}|\mathcal{G}(0,j+1)|\, |f(j,y(j,k,\eta))-f(j,y(j,k,\tilde{\eta}))|,\\                                   &\leq &2\sum_{j=0}^{\infty}|\mathcal{G}(0,j+1)\mu(j)|<\infty,
\end{eqnarray*}
thus, given $\varepsilon>0$ there exists $J>0$ such that
\begin{displaymath}
 \sum_{j=J+1}^{\infty}|\mathcal{G}(0,j+1)|\, |f(j,y(j,k,\eta))-f(j,y(j,k,\tilde{\eta}))|< \frac{\varepsilon}{3\,||\Phi(k,0)||}
\end{displaymath}
and we have that
$$
|w^{*}(0;(k,\eta))-w^{*}(0;(k,\tilde{\eta}))|\leq \sum_{j=0}^{J}|\mathcal{G}(0,j+1)|\, |f(j,y(j,k,\eta))-f(j,y(j,k,\tilde{\eta}))|+\frac{\varepsilon}{3\, ||\Phi(k,0)||}.
$$

In addition, by {\bf (P3)} its follows that
$$
|w^{*}(0;(k,\eta))-w^{*}(0;(k,\tilde{\eta}))|\leq \sum_{j=0}^{J}|\mathcal{G}(0,j+1)\gamma(j)|\, |y(j,k,\eta)-y(j,k,\tilde{\eta})|+\frac{\varepsilon}{3\, ||\Phi(k,0)||}.
$$
Moreover, it is easy to see that if $j>k$, the discrete Gronwall's inequality (see for example \cite{Elaydi,Popenda}) implies
\begin{equation}
\label{cont-LG}
|y(j,k,\eta)-y(j,k,\tilde{\eta})|\leq |\eta-\tilde{\eta}|\prod_{p=k}^{j-1}(||A(p)-I||+\gamma(p)).
\end{equation}

By using (\ref{conP5rev}) and (\ref{cont-LG}) we have that
\begin{eqnarray}
\label{acotamientow*}
|w^{*}(0;(k,\eta))-w^{*}(0;(k,\tilde{\eta}))|&\leq & |\eta-\tilde{\eta}|\, \mathcal{B}_{k}(J)+\frac{\varepsilon}{3\, ||\Phi(k,0)||},
\end{eqnarray}
where
$$
\mathcal{B}_{k}(J)=\sum_{j=0}^{J}\mathcal{A}_{k}(j).
$$
where $\mathcal{A}_{k}(j)$ is defined by
\begin{displaymath}
    \mathcal{A}_{k}(j)=\left\{\begin{array}{clr}
    |\mathcal{G}(0,j+1)\gamma(j)|\, \mathcal{C}_{k}(j) &\textnormal{if} & j<k,\\\\
   |\mathcal{G}(0,j+1)\gamma(j)|\displaystyle \prod_{p=k}^{j-1}(||A(p)-I||+\gamma(p))& \textnormal{if} & j\geq k,
    \end{array}\right.
    \end{displaymath}

Finally if 
$$
\displaystyle \delta(\varepsilon,\mu,k)=\min\left\{\frac{\varepsilon}{3\, \mathcal{C}_{k}(0)\, ||\Phi(k,0)||},\frac{\varepsilon}{3\, \mathcal{B}_{k}(J)\,||\Phi(k,0)||}\right\},\bigskip
$$ 
and $|\eta-\tilde{\eta}|<\delta,$ from $(\ref{conP5revparticular})$ and $(\ref{acotamientow*})$ it follows that
\bigskip
\begin{eqnarray*}
|G(k,\eta)-G(k,\tilde{\eta})|&\leq & ||\Phi(k,0)||\, \left\{|y(0,k,\eta)-y(0,k,\tilde{\eta})|+|w^{*}(0;(k,\eta))-w^{*}(0;(k;\tilde{\eta}))|\right\},\\\\
                          &\leq & ||\Phi(k,0)||\, \left\{\mathcal{C}_{k}(0)|\eta-\tilde{\eta}|+|\eta-\tilde{\eta}|\, \mathcal{B}_{k}(J)+\frac{\varepsilon}{3\, ||\Phi(k,0)||}\right\},\\\\
                          &< &||\Phi(k,0)||\, \left\{\frac{\varepsilon}{3\, ||\Phi(k,0)||}+\frac{\varepsilon}{3\, ||\Phi(k,0)||}+\frac{\varepsilon}{3\, ||\Phi(k,0)||}\right\},\\
                          &=&\varepsilon.
\end{eqnarray*}

Therefore, for each $k\in \mathbb{Z}^{+}$ we have that $\mu \mapsto G(k,\eta)$ is a continuous map. 

\bigskip

\noindent\textit{Step 6: $H$ is a continuous map.} In order to study the continuity of (\ref{Homeo-H}) for any fixed $k\in \mathbb{Z}^{+}$, we will prove that
$$
\xi \mapsto z^{*}(k;(k,\xi))
$$
is continuous for any $k\in \mathbb{Z}^{+}$ since $H(k,\xi)=\xi+z^{*}(k;(k,\xi))$.

First of all, let us consider the family of sequences $\{z_{n}(j;(k,\xi))\}_{(k,\xi)} \in \ell_{\infty}(\mathbb{Z}^{+})$ defined recursively as follows:
\begin{displaymath}
\begin{array}{rcl}
z_{n+1}(j;(k,\xi)) & = &
\sum\limits_{\ell=0}^{+\infty}\mathcal{G}(j,\ell+1)f(\ell,x(\ell,k,\xi)+z_{n}(\ell;(k,\xi)))  \quad n\geq 0 \\\\
z_{0}(j;(k,\xi))&=& \varphi \in \ell_{\infty},\quad \textnormal{for all $\xi \in \mathbb{R}^{d}$}.
\end{array}
\end{displaymath}

Now, by \textbf{(P4)}--\textbf{(P5)} it is easy to verify that for each $\xi\in \mathbb{R}^{d}$ and $n\geq m$ we have that
\begin{eqnarray*}
|z_{n}(j;(k,\xi))-z_{m}(j;(k,\xi))| &\leq &(q^{n-1}+\ldots +q^{m})|z_{1}(j;(k,\xi))-\varphi(j)|, \\\\
  &\leq &q^{m}(1-q)^{-1}|z_{1}(j;(k,\xi))-\varphi(j)|, \\\\
&\leq & q^{m}(1-q)^{-1}(|\varphi|_{\infty}+p).
\end{eqnarray*}
Thus, its follows that $n\mapsto z_{n}(j;(k,\xi))$ is a Cauchy sequence, that is
\begin{displaymath}
\forall\tilde{\varepsilon}\,\,\exists N(\tilde{\varepsilon})>0 \quad \textnormal{such that} \quad n,m>N \Rightarrow |z_{n}(j;(k,\xi))-z_{m}(j;(k,\xi))|<\tilde{\varepsilon},
\end{displaymath}
and since $N$ is not dependent of the parameter $\xi\in \mathbb{R}^{d}$ then we have that 
\begin{eqnarray*}
\lim\limits_{n\to +\infty}z_{n}(j;(k,\xi))=z^{*}(j;(k,\xi))  = \sum\limits_{\ell=0}^{+\infty}\mathcal{G}(j,\ell+1)f(\ell,x(\ell,k,\xi)+z^{*}(\ell;(k,\xi))) 
\end{eqnarray*}
is uniform with respect to $j\in \mathbb{Z}^{+}$ and $(k,\xi)\in \mathbb{Z}\times\mathbb{R}^{d}$. Hence
\begin{equation}
\label{limit-uniform}
    \forall\tilde{\varepsilon}\,\,\exists N(k,\varepsilon)>0 \quad \textnormal{such that} \quad n>N \Rightarrow |z_{n}(j;(k,\xi))-z^{*}(j;(k,\xi))|<\tilde{\varepsilon}.
\end{equation}

Secondly, we will see that the map $\xi\mapsto z_{n}(j;(k,\xi))$ is continuous for any $\xi\in \mathbb{R}^{d}$ and $n\in \mathbb{N}.$ For this, let us prove by induction that for each $n\in \mathbb{N}$ and $\varepsilon>0$ there exists $\delta_{n}(\varepsilon,\xi,k)>0$ such that
\begin{eqnarray}
\label{induccion}
|z_{n}(j;(k,\xi))-z_{n}(j;(k;\tilde{\xi}))|<\varepsilon,\quad \textnormal{if} \quad |\xi-\tilde{\xi}|<\delta_{n}.
\end{eqnarray}

From this, we have that $|z_{0}(j;(k,\xi))-z_{0}(j;(k,\tilde{\xi}))|=0<\varepsilon.$ Now, let us suppose that $(\ref{induccion})$ is verified for some $n\in \mathbb{N}.$ In addition, from \textbf{(P4)} we know that, given $j\in \mathbb{Z}^{+}$, $\varepsilon>0$ and $\theta\in (0,1)$ such that
$\theta+q<1$, there exists $L(\varepsilon)>0$ such that
\begin{eqnarray}
\label{colapequena}
 2\sum\limits_{\ell=L+1}^{\infty}||\mathcal{G}(j,\ell+1)||\mu(\ell)<\theta\varepsilon.
\end{eqnarray} 

Thus, if 
$$
\delta_{n+1}(\varepsilon,\xi,k)=\min\left\{\delta_{n}(\varepsilon,\xi,k),\frac{\varepsilon(1-\{q+\theta\})}{\mathcal{S}_{k}q}\right\},
$$
where by {\bf (P1)} we can define
$$
\mathcal{S}_{k}=\max\{||\Phi(0,k)||,||\Phi(1,k)||,...,||\Phi(L,k)||\}<+\infty.
$$

By using $(\ref{colapequena})$ combined with the inductive hypothesis we have that
\begin{eqnarray*}
|z_{n+1}(j;(k,\xi))-z_{n+1}(j;(k;\tilde{\xi}))|&\leq &\sum_{\ell=0}^{L}||\mathcal{G}(j,\ell+1)||\, \gamma(\ell)\left\{||\Phi(\ell,k)||\, |\xi-\tilde{\xi}|+\varepsilon\right\}+\theta \varepsilon\\\\
                &< & \delta_{n+1}\mathcal{S}_{k}q+\varepsilon q+\theta \varepsilon\\
                &= & \varepsilon.
\end{eqnarray*}

Hence, we can concluded that $\xi\mapsto z_{n}(j;(k,\xi))$ is a continuous map.

Subsequently, we will see that the map $\xi \mapsto z^{*}(j;(k,\xi))$ is continuous for any $j\in \mathbb{Z}^{+}$ and $(k,\xi) \in \mathbb{Z}^{+}\times\mathbb{R}^{d}.$ For this, from (\ref{limit-uniform}), we know that for each $\varepsilon>0$ there exists $N(\varepsilon)\in \mathbb{N}$ such that
$$
|z^{*}(j;(k,\tilde{\xi}))-z_{N}(j;(k,\tilde{\xi}))|<\frac{\varepsilon}{3} \quad \textnormal{and} \quad |z^{*}(j;(k,\xi))-z_{N}(j;(k,\xi))|<\frac{\varepsilon}{3},
$$
for any $j\in \mathbb{N}$. In addition, there exists $\delta_{N}(\varepsilon,\xi,k)>0$ such that
$$
|z_{N}(j;(k,\xi))-z_{N}(j,(k,\tilde{\xi}))|<\frac{\varepsilon}{3},\quad \textnormal{if},\quad |\xi-\tilde{\xi}|<\delta_{N}.
$$

Thus, from the above inequalities we have that if $|\xi-\tilde{\xi}|<\delta_{N}$ then
\begin{displaymath}
\begin{array}{rcl}
|z^{*}(j;(k,\xi))-z^{*}(j;(k,\tilde{\xi}))|&\leq & \,\,\,\,\,|z^{*}(j;(k,\xi))-z_{N}(j;(k,\xi))| \\\\
&&+|z_{N}(j;(k,\xi))-z_{N}(j;(k,\tilde{\xi}))|\\\\
&&+|z_{N}(j;(k,\tilde{\xi}))-z^{*}(j;(k,\tilde{\xi}))|\\\\
&<& \varepsilon.
\end{array}
\end{displaymath}

Finally, since $\xi\mapsto z^{*}(j;(k,\xi))$ is a continuous map for all $j\in \mathbb{N},$ in particular $\xi\mapsto z^{*}(k;(k,\xi))$ is a continuous map too. Therefore the map
$$
\xi\mapsto H(k,\xi)=\xi+z^{*}(k;(k,\xi))
$$
is continuous.
\end{proof}
\subsection{Some Consequences}

The intermediate computations proving Theorem \ref{main} allow us to verify some interesting identities:

\begin{remark}
The maps $w^{*}$ and $z^{*}$ defined respectively by (\ref{map1}) and (\ref{FP}) verify the following identities:
\begin{displaymath}
\begin{array}{rcl}
w^{*}(k;(m,\eta))+z^{*}(k;(m,G(m,\eta)))&=&0,\\\\
z^{*}(k;(m,\xi))+w^{*}(k;(m,H(m,\xi))) &=&0.
\end{array}
\end{displaymath}

In fact, by using the identity $H[k,G[k,y(k,m,\eta)]]=y(k,m,\eta)$ combined with (\ref{Homeo-H}), the first and second identities of (\ref{identity-GX}) and (\ref{identidad-dif}) we have that:
\begin{displaymath}
\begin{array}{lcl}
w^{*}(k;(m,\eta))+z^{*}(k;(k,G[k,y(k,m,\eta)])) &=&0\\\\
w^{*}(k;(m,\eta))+z^{*}(k;(k,x(k,m,G(m,\eta)])) &=&0\\\\
w^{*}(k;(m,\eta))+z^{*}(k;(m,G(m,\eta))) &=&0,
\end{array}
\end{displaymath}
and the first identity follows. The second identity can be deduced by replacing $\eta$ by $H(m,\xi)$ in the first identity
and using $G(m,H(m,\xi))=\xi$.
\end{remark}

By considering $m=k$ in the above identities, we have the following consequence
\begin{corollary}
The maps $G$ and $H$ of the  $\mathbb{Z}^{+}$--topological equivalence satisfies the fo\-llo\-wing fixed point properties:
\begin{displaymath}
\begin{array}{rcl}
G(k,\eta)&=&\eta-z^{*}(k;(k,G(k,\eta))),\\\\
H(k,\xi)&=&\xi-w^{*}(k;(k,H(k,\xi))).
\end{array}
\end{displaymath}
\end{corollary}

On the other hand, Theorem \ref{main} has some interesting byproducts, the first one follows immediately after the fact that the topological equivalence is an equivalence relation:
\begin{corollary}
If the linear system (\ref{lineal}) satisfies the assumptions {\bf (P1)}-{\bf (P2)}, then for any function
$g\colon \mathbb{Z}^{+}\times \mathbb{R}^{d}\to \mathbb{R}^{d}$ satisfying the assumptions {\bf (P3)}-{\bf (P6)} it follows that the nonlinear systems
\begin{eqnarray*}
y_{k+1}=A(k)y_{k}+g(k,y_{k})
\end{eqnarray*}
and (\ref{nolinealsystem}) are $\mathbb{Z}^{+}-$ topologically equivalent.
\end{corollary}

A second byproduct of Theorem \ref{main} considers the case where a projector $P(n)=I$ for any $\mathbb{Z}^{+}$, that is, the system (\ref{lineal}) is a nonuniform contraction. Nevertheless this result has interest on itself ans should be treated separately.

\begin{corollary}
If the properties \textnormal{\textbf{(P1)}}, \textnormal{\textbf{(P3)}}, \textnormal{\textbf{(P6)}} 
and
\begin{itemize}
\item[{\bf (S1)}] There exists a bounded sequence $\rho$ and a decreasing sequence $h$ convergent to zero with $h(0)=1$ such that:
    \begin{displaymath}
\begin{array}{rl}
    ||\Phi(k,n)||\leq \displaystyle \rho(n) \frac{h(k)}{h(n)}, & \forall k\geq n \geq 0.
    \end{array}
    \end{displaymath}
    \item[{\bf{(S2)}}] The sequences $\rho$ and $h$ defined above verify 
\begin{displaymath}
     \sum\limits_{j=0}^{k-1}\mu(j)\rho(j+1)\frac{h(k)}{h(j+1)}<\infty \quad \textnormal{for any $k\in \mathbb{Z}^{+}$}.
\end{displaymath}

\item[{\bf{(S3)}}] The sequences $\rho$ and $h$ defined above verify 
\begin{displaymath}
     \sum\limits_{j=0}^{k-1}\gamma(j)\rho(j+1)\frac{h(k)}{h(j+1)}:=q <1 \quad \textnormal{for any $k\in \mathbb{Z}^{+}$}.
\end{displaymath}
\end{itemize}
are satisfied then the systems (\ref{lineal}) and (\ref{nolinealsystem}) are  $\mathbb{Z}^{+}$--topologically equivalent with
\begin{equation}
\label{homeo-G2}
\left\{\begin{array}{rcl}
G(k,\xi)&=&x(k,0,y(0,k,\xi))=\Phi(k,0)y(0,k,\xi),\\\\
H(k,\xi)&=&y(k,0,x(0,k,\xi)).
\end{array}\right.
\end{equation}
\end{corollary}

\begin{proof}
The topological equivalence is immediate since \textbf{(S1)},\textbf{(S2)} and \textbf{(S3)} are particular cases of \textbf{(P2)},\textbf{(P4)} and \textbf{(P5)} respectively. Nevertheless, we can gain more insight about $G$ and $H$. In fact, as $Q(n)=0$ for any $n\in \mathbb{Z}^{+}$
we have the identity (\ref{homeo-G1}). Moreover, we can easily prove that $\xi\mapsto G(k,\xi)$ is a bijection for any $k\in \mathbb{Z}^{+}$. In fact, the injectiveness
is a straightforward consequence of the uniqueness of solutions. On the other hand, given an arbitrary $z\in \mathbb{R}^{d}$ it is easy to see that $G(k,\xi)=z$
 with $\xi=y(k,0,\Phi(0,k)z)$ and the surjectiveness follows.

As we know that
\begin{equation}
\label{inversaH}
H(k,\xi)=G^{-1}(k,\xi)
\end{equation}
for any $k\in \mathbb{Z}^{+}$, from $(\ref{Homeo-G})$ we have that
$$
G(k,H(k,\xi))=\Phi(k,0)y(0,k,H(k,\xi))=\xi
$$
or equivalently
\begin{displaymath}
y(0,k,H(k,\xi))=\Phi(0,k)\xi=x(0,k,\xi).
\end{displaymath}

In addition, from $(\ref{identity-HY})$ combined with the identity above
we have that
$$
H(k,\xi)=y(k,0,y(0,k,H(k,\xi))=y(k,0,x(0,k,\xi))
$$ 
and the result follows.
\end{proof}

\begin{remark}
We point out the remarkable simplicity of the identities (\ref{homeo-G2}) compared with previous references \cite{Castaneda2016,Reinfelds1997,Reinfelds}. In addition, we emphasize its novelty for the discrete case.

In addition, if we assume that for any fixed $k\in \mathbb{Z}^{+}$ the 
map $x \mapsto f(k,x)$ and its derivatives up to order $r$--th are continuous with $r\geq 1$. It can be proved that the map $\xi \mapsto G(k,\xi)$ is of class $C^{r}$ for any fixed $k\in \mathbb{Z}^{+}$.

In fact by using (\ref{dif-1}), it can be proved that
$\xi \mapsto \frac{\partial y}{\partial \xi}(n,k,\xi)$ is well defined for any $0\leq n< k-1$ and is solution of the matrix difference equation
\begin{displaymath}
\left\{\begin{array}{rcl}
z_{n+1}&=&[A(n)+Df(n,y(n,k,\mu))]z_{n}\\\\
z_{k} &=& I
\end{array}\right.
\end{displaymath}
and by using the first identity of (\ref{homeo-G2}) we have that 
$$
\frac{\partial G}{\partial \xi}(k,0)=\Phi(k,0)\frac{\partial y}{\partial \xi}(0,k,\xi) \quad \textnormal{for any $k\in \mathbb{Z}^{+}$}
$$
and the higher order derivatives can be computed directly from the above identity.
\end{remark}

\section{Topological equivalence and asymptotic stability}
This section is focused in the special case that the linear system (\ref{lineal}) is asymptotically stable (not necessarily uniform) and studies the preservation of stability by topological equivalence.

\begin{definition}
The solution $y^{*}$ is an equilibrium of the nonlinear system $(\ref{nolinealsystem})$ if
\begin{equation}
\label{equilibrium}
y^{*}=A(n)y^{*}+f(n,y^{*}) \quad \textnormal{for any $n\in \mathbb{Z}^{+}$}.
\end{equation}
\end{definition}

\begin{theorem}
\label{preservacion}
If the properties \textnormal{\textbf{(P1)}},\textnormal{\textbf{(P3)}},\textnormal{\textbf{(P6)}},\textnormal{\textbf{(S1)}--\textbf{(S3)}} 
and
\begin{itemize}
\item[{\bf{(S4)}}] The sequences $\rho$ and $h$ defined above verify 
$$
\lim\limits_{k\to +\infty}
h(k)
\prod\limits_{j=0}^{k-1}\left(1+\gamma(j) \rho(j+1)\frac{h(j)}{h(j+1)}\right)=0
$$
\end{itemize}
are satisfied then
\begin{itemize}
    \item[(i)] If the system (\ref{nolinealsystem}) has an equilibrium $y^{*}$, then it is unique,
    \item[(ii)] If $y^{*}=0$, that is $f(n,0)=0$ for any $n\in \mathbb{Z}^{+}$. Then:
    $$
    H(k,0)=G(k,0)=0 \quad \textnormal{for any $k\in \mathbb{Z}^{+}$},
    $$
    \item[(iii)] If $y^{*}\neq 0$, then 
    \begin{displaymath}
    \lim\limits_{k\to +\infty}H(k,0)=y^{*} \quad \textnormal{and} \quad \lim\limits_{k\to +\infty}G(k,y^{*})=0.
    \end{displaymath}
\end{itemize}

\end{theorem}

\begin{proof}
Ir order to prove (i) notice that if $y^{*}$ and $\bar{y}$ are equilibria of (\ref{nolinealsystem}) then it follows that
$y(n,0,y^{*})=y^{*}$ and $y(n,0,\bar{y})=\bar{y}$ for any $n\in \mathbb{Z}^{+}$, this is equivalent to:
\begin{displaymath}
y^{*}-\bar{y}=\Phi(n,0)(y^{*}-\bar{y})+\sum\limits_{k=0}^{n-1}\Phi(n,k+1)\{f(k,y^{*})-f(k,\bar{y})\} \quad \textnormal{for any $n\in \mathbb{Z}^{+}$}.
\end{displaymath}

Then, it follows by \textbf{(S1)} and \textbf{(S3)} that
\begin{displaymath}
\begin{array}{rcl}
|y^{*}-\bar{y}| & \leq & ||\Phi(n,0)|||y^{*}-\bar{y}|+\sum\limits_{k=0}^{n-1}||\Phi(n,k+1)|||f(k,y^{*})-f(k,\bar{y})| \\\\
&\leq  & \displaystyle \rho(0)h(n)|y^{*}-\bar{y}|+\sum\limits_{k=0}^{n-1}\gamma(k)\rho(k+1)\frac{h(n)}{h(k+1)} |y^{*}-\bar{y}| \\\\
&\leq  & \displaystyle  \rho(0)h(n)|y^{*}-\bar{y}|+q|y^{*}-\bar{y}|.
\end{array}
\end{displaymath}

We will see that $y^{*}=\bar{y}$. Indeed otherwise $|y^{*}-\bar{y}|\neq 0$ which combined with the above inequality leads to
\begin{displaymath}
1\leq  \rho(0)h(n)+q \quad \textnormal{for any $n\in\mathbb{Z}^{+}$}.
\end{displaymath}

Now, letting $n\to +\infty$ and using \textbf{(S1)} we obtain a contradiction  with \textbf{(S3)} and the uniqueness of the equilibrium follows. Subsequently from $(\ref{homeo-G1})$ we have that
$$
G(k,y^{*})=\Phi(k,0)y(0,k,y^{*})=\Phi(k,0)y^{*}.
$$

Thus if $y^{*}=0$ we have that $G(k,0)=0$. On the other hand, from the fact that $H(k,G(k,0))=0$ we can deduce that $H(k,0)=0$ for any $k\in \mathbb{Z}^{+}$ and thus (ii) has been proved.

Finally if $y^{*}\neq 0$ then
$$
|G(k,y^{*})|=|\Phi(k,0)y^{*}|\leq ||\Phi(k,0)||\, |y^{*}|\leq \rho(0)h(k)|y^{*}|, \quad \forall k\in \mathbb{Z}^{+} 
$$
Thus, letting $k\to +\infty$ we conclude that $\lim\limits_{k\to +\infty} G(k,y^{*})=0$.

Now from $(\ref{Homeo-H})$ combined with the fact that $y^{*}$ is an equilibrium we have that
\begin{eqnarray*}
|H(k,0)-y^{*}|&\leq &|\Phi(k,0)y^{*}|+\sum_{j=0}^{k-1}|\Phi(j,k+1)|\, |f(j,z^{*}(j;(k,0))+x(j,k,0))-f(j,y^{*})|,\\
              &\leq & |\Phi(k,0)y^{*}|+\sum_{j=0}^{k-1}\gamma(j) |\Phi(k,j+1)|\, |z^{*}(j;(k,0))+x(j,k,0)-y^{*}|,\\
              &\leq & ||\Phi(k,0)||\,|y^{*}|+\sum_{j=0}^{k-1}\gamma(j) |\Phi(k,j+1)|\, |H(j,0)-y^{*}|,\\
              &\leq & \rho(0)h(k) |y^{*}|+\sum_{j=0}^{k-1}\gamma(j)\rho(j+1)\frac{h(k)}{h(j+1)}|H(j,0)-y^{*}|,\\
\end{eqnarray*}
which is the equivalent to
\begin{displaymath}
\begin{array}{rcl}
\frac{1}{h(k)}|H(k,0)-y^{*}|  &\leq &  \displaystyle \rho(0) |y^{*}|+\sum\limits_{j=0}^{k-1}\gamma(j) \rho(j+1)\frac{1}{h(j+1)}|H(j,0)-y^{*}|,\\\\
 &\leq & \displaystyle \rho(0) |y^{*}|+\sum\limits_{j=0}^{k-1}\gamma(j) \rho(j+1)\frac{h(j)}{h(j+1)}\frac{1}{h(j)}|H(\ell-1,0)-y^{*}|,
\end{array}
\end{displaymath}

Later if $W(k)=\frac{1}{h(k)}|H(k,0)-y^{*}|$ then
$$
W(k)\leq \rho(0)|y^{*}|+\sum\limits_{j=0}^{k-1}\gamma(j) \rho(j+1)\frac{h(j)}{h(j+1)}W(j).
$$

Now from the discrete Gronwall inequality (see for example \cite{Elaydi,Popenda}) we have that
$$
W(k)\leq \rho(0)|y^{*}|
\prod\limits_{j=0}^{k-1}\left(1+\gamma(j) \rho(j+1)\frac{h(j)}{h(j+1)}\right).
$$

Thus
$$
|H(k,0)-y^{*}|\leq \rho(0)|y^{*}|\,h(k)
\prod\limits_{j=0}^{k-1}\left(1+\gamma(j) \rho(j+1)\frac{h(j)}{h(j+1)}\right).
$$

Therefore, the property (iii) follows from \textbf{(S4)}.
\end{proof}

\begin{corollary}
Under the assumptions of Theorem \ref{preservacion}, if the nonlinear system
(\ref{nolinealsystem}) has an equilibrium $y^{*}$, then is asymptotically stable.
\end{corollary}

\begin{proof}
We will only consider the case $y^{*}\neq 0$ since the other one is simpler. Let
$k\mapsto y(k,j,\eta)$ be a solution of (\ref{nolinealsystem}). By using 
identity (\ref{identity-HY}) we know that there exists
a unique $\xi\in \mathbb{R}^{d}$ such that $H(j,\xi)=\eta$ and $y(k,j,\eta)=H[k,x(k,j,\xi)]$. Now, 
notice that
$$
|y(k,j,\eta)-y^{*}|=|H[k,x(k,j,\xi)]-y^{*}|\leq |H[k,x(k,j,\xi)]-H[k,0]|+|H[k,0]-y^{*}|
$$
and the result follows by proving that the two right terms are small enough for big values of $k$. 

First of all, by using \textbf{(S1)} we know that any 
solution $k\mapsto x(k,j,\xi)$ of the linear system (\ref{lineal}) verifies the
asymptotic behavior $\lim\limits_{k\to +\infty}x(k,j,\xi)=0$. 

The above limit combined with the continuity of the
map $x \mapsto H(k,x)$ for any fixed $k$ implies that $|H[k,x(k,j,\xi)]-H[k,0]|$ 
is arbitrarily small for any fixed (but big enough) value of $k$. 

Moreover, by using the statement (iii) of Theorem \ref{preservacion} we have the asymptotic behavior
$\lim\limits_{k\to +\infty}H(k,0)=y^{*}$ and the result follows.
\end{proof}

In the particular case that the linear system (\ref{lineal}) is uniformly asymptotically stable,
namely, it has and exponential dichotomy on $\mathbb{Z}^{+}$ with projectors $P(n)=I$ and $Q(n)=0$ (see Remark \ref{stabilite}), the assumption \textbf{(S4)} can be dropped and the above result becomes simpler:
\begin{corollary}
Assume that the properties \textnormal{\textbf{(P1)}} 
and \textnormal{\textbf{(P6)}} are verified, the property \textnormal{\textbf{(S2)}}
is satisfied with $\rho(n):=K>0$ and $h(n)=\theta^{n}$ with $\theta\in (0,1)$ and the property \textnormal{\textbf{(P3)}}
is verified constants sequences $\gamma(n):=\gamma$ and $\mu(n):=\mu$ which satisfy
\begin{equation}
\label{mu+exp}
\frac{\mu K}{1-\theta}<\infty   \quad \textnormal{and} \quad \frac{\gamma K}{1-\theta}<q<1
\end{equation}
then we have that
\begin{itemize}
    \item[(i)] If the system (\ref{nolinealsystem}) has an equilibrium $y^{*}$, then it is unique,
    \item[(ii)] If $y^{*}=0$, that is $f(n,0)=0$ for any $n\in \mathbb{Z}^{+}$. Then:
    $$
    H(k,0)=G(k,0)=0 \quad \textnormal{for any $k\in \mathbb{Z}^{+}$},
    $$
    \item[(iii)] If $y^{*}\neq 0$, then 
    \begin{displaymath}
    \lim\limits_{k\to +\infty}H(k,0)=y^{*} \quad \textnormal{and} \quad \lim\limits_{k\to +\infty}G(k,y^{*})=0.
    \end{displaymath}
\end{itemize}
\end{corollary}

\begin{proof}
It is straightforward to see that the inequalities (\ref{mu+exp}) implies
\textbf{(S2)} and \textbf{(S3)}. Now, we will see that (\ref{mu+exp}) also implies
\textbf{(S4)}.

In fact, notice that the right inequality of (\ref{mu+exp}) implies that $0<\theta+\gamma K<1$.
Moreover we can see that
\begin{eqnarray*}
h(k)\prod_{j=0}^{k-1}1+\gamma(j)\rho(j+1)\frac{h(j)}{h(j+1)}&=&\theta^{k}\prod_{j=0}^{k-1}\left(1+\frac{K\gamma}{\theta}\right)\\\\
&=& \theta^{k}\prod\limits_{j=0}^{k-1}\frac{\theta+\gamma K}{\theta}\\\\
&=& (\theta+\gamma K)^{k}
\end{eqnarray*}
and {\bf (S4)} follows since  $0<\theta+\gamma K<1$.
\end{proof}

\end{document}